\documentclass[12pt,reqno]{amsart}

\usepackage{amssymb,hyperref}

\newtheorem{theorem}{Theorem}[section]
\newtheorem{proposition}[theorem]{Proposition}
\newtheorem{lemma}[theorem]{Lemma}
\newtheorem{corollary}[theorem]{Corollary}

\theoremstyle{definition}
\newtheorem{definition}[theorem]{Definition}
\newtheorem{remark}[theorem]{Remark}

\numberwithin{equation}{section}

\begin{document}

\baselineskip=15pt

\title[Branched Holomorphic Cartan Geometries]{Branched Holomorphic Cartan 
Geometries and Calabi--Yau manifolds}

\author[I. Biswas]{Indranil Biswas}

\address{School of Mathematics, Tata Institute of Fundamental
Research, Homi Bhabha Road, Mumbai 400005, India}

\email{indranil@math.tifr.res.in}

\author[S. Dumitrescu]{Sorin Dumitrescu}

\address{Universit\'e C\^ote d'Azur, CNRS, LJAD, France}

\email{dumitres@unice.fr}

\subjclass[2010]{53B21, 53C56, 53A55}

\keywords{Complex projective structure, branched Cartan geometry, Calabi--Yau 
manifold, complex surface.}

\date{}

\begin{abstract}
We introduce the concept of a {\it branched holomorphic Cartan geometry}. It 
generalizes to higher dimension the definition of branched (flat) complex 
projective structure on a Riemann surface introduced by Mandelbaum \cite{M1}.
This new framework is much more flexible than that of
the usual holomorphic Cartan geometries. We show that all compact 
complex projective manifolds admit a branched flat holomorphic projective structure. 
We also give an example of a non-flat branched holomorphic normal projective 
structure on a compact complex surface. It is known that no compact complex surface admits such a
structure with empty branching 
locus. We prove that non-projective compact simply connected K\"ahler 
Calabi--Yau manifolds do not admit any branched holomorphic projective structures. The 
key ingredient of its proof is the following result of independent interest: {\it If $E$ 
is a holomorphic vector bundle over a compact simply connected K\"ahler Calabi--Yau 
manifold, and $E$ admits a holomorphic connection, then $E$ is a trivial holomorphic 
vector bundle, and any holomorphic connection on $E$ is trivial.}
\end{abstract}

\maketitle

\tableofcontents

\section{Introduction}

The uniformization theorem for Riemann surfaces asserts that any Riemann surface is 
isomorphic either to the projective line ${\mathbb C}P^1$, or to a quotient of 
$\mathbb C$, or of the unit disk in $\mathbb C$, by a discrete group of projective 
transformations (lying in the M\"obius group $\text{PGL}(2, {\mathbb C})$). In 
particular, any Riemann surface $X$ admits a holomorphic atlas with coordinates in 
${\mathbb C}P^1$ and transition maps in $\text{PGL}(2, {\mathbb C})$. This defines 
a {\it (flat) complex projective structure} on $X$. Complex projective structure
on Riemann surfaces were introduced in connection with the study of the second order ordinary 
differential equations on complex domains and they had a very major role to play in 
understanding the framework of uniformization theorem \cite{Gu,StG}.

The complex projective line acted on by the M\"obius group is a geometry in the 
sense of Klein's Erlangen program in which he proposed to study Euclidean, affine and 
projective geometries in the unifying frame of the homogeneous model spaces $G/H$, 
where $G$ is a finite dimensional Lie group and $H$ a closed subgroup in $G$.

Following Ehresmann \cite{Eh}, a manifold $X$ is locally modelled on a homogeneous 
space $G/H$, if $X$ admits an atlas with charts in $G/H$ and transition maps given 
by elements in $G$ using its left-translation action on
$G/H$. Any $G$-invariant geometric feature of $G/H$ will have an 
intrinsic meaning on $X$.

Elie Cartan generalized Klein's homogeneous model spaces to {\it Cartan geometries} 
(or {\it Cartan connections}) (see definition in Section \ref{s2.1}). We recall 
that these are geometrical structures infinitesimally modelled on homogeneous 
spaces $G/H$. A Cartan geometry associated to the affine (respectively, 
projective) geometry is classically called an affine (respectively, projective) 
connection. A Cartan geometry on a manifold $X$ is equipped with a curvature tensor 
(see definition in Section \ref{s2.1}) which vanishes exactly when $X$ is locally 
modelled on $G/H$ in the sense of Ehresmann \cite{Eh}. In such a situation the 
Cartan geometry is called {\it flat}.

In this article we study holomorphic Cartan geometries on compact complex manifolds 
of complex dimension at least two. Contrary to the situation of Riemann surfaces, 
holomorphic Cartan geometries in higher dimension are not always flat. Moreover, 
for a compact complex manifold, to admit a holomorphic Cartan geometry is a very 
stringent condition: indeed most of the compact complex manifolds do not admit any holomorphic 
Cartan geometry.

In \cite{KO}, Kobayashi and Ochiai proved that compact complex surfaces admitting a 
holomorphic projective connection are biholomorphic either to the complex 
projective plane ${\mathbb C}P^2$, or to a quotient of an open set in ${\mathbb 
C}P^2$ by a discrete group of projective transformations acting properly and 
discontinuously on it. In particular, such a compact complex
surface also admits a flat complex projective structure 
(modelled on ${\mathbb C}P^2$). In this list of compact complex surfaces admitting 
(flat) complex projective structures, the only {\it projective} ones are ${\mathbb 
C}P^2$, abelian varieties (and their unramified finite quotients) and quotients of 
the ball (complex hyperbolic plane).

Another source of inspiration for this paper is the work of Mandelbaum \cite{M1, M2} who 
introduced and studied {\it branched affine and projective structures} on Riemann 
surfaces. According to his definition, branched projective structures on Riemann 
surfaces are given by some holomorphic atlas where local charts are finite branched 
coverings on open sets in ${\mathbb C}P^1$ and transition maps lie in 
$\text{PGL}(2, {\mathbb C})$. Such structures arise naturally in the study of
conical hyperbolic structures, and also when one consider ramified coverings.

Here we define a more general notion of {\it branched holomorphic Cartan geometry} 
on a complex manifold $X$ (see Definition \ref{def1}), which is valid also in higher 
dimension and encompass non-flat geometries. We show that the notion of curvature
continues to hold, and in fact the curvature
vanishes exactly when there is a holomorphic atlas where local charts are 
branched holomorphic maps to the model $G/H$. Two local charts agree up to the 
action on $G/H$ of an element in $G$. The geometric description of the flat case follows the 
description in the usual case: there exists a branched holomorphic developing map from the universal 
cover of $X$ to the model $G/H$ which is a local biholomorphism away from a 
divisor. This developing map is equivariant with respect to the monodromy 
homomorphism (which is a group homomorphism from the fundamental group of $X$ into 
$G$, unique up to post-composition by inner automorphisms of $G$).

This new notion of branched Cartan geometry is much more flexible than the usual 
one. For example, all compact complex projective manifolds admit a branched flat 
holomorphic projective structure (see Proposition \ref{algebraic proj struct}).

We also prove that there exist branched {\it normal} holomorphic projective 
connections (see definition in Section \ref{s2.2}) on compact surfaces which are 
not flat (see Proposition \ref{normal}). This is not the case for holomorphic 
projective connections with empty branching set, meaning any normal projective
structure on a compact complex surface is known to be automatically flat \cite{D2}.

The following is proved in Theorem \ref{thm2}:

{\it If $E$ is a holomorphic 
vector bundle over a compact simply connected K\"ahler Calabi--Yau manifold, and $E$ 
admits a holomorphic connection, then $E$ is a trivial holomorphic vector bundle 
equipped with the trivial connection.}

This result, which is of independent interest, is related to the classification of 
branched holomorphic Cartan geometries on Calabi--Yau manifolds. It yields 
Corollary \ref{final} that asserts the following:

\begin{enumerate}
\item \textit{Any branched holomorphic Cartan geometry of type $G/H$, with $G$ complex affine Lie group, on a compact simply connected (K\"ahler) 
Calabi--Yau manifold is flat. Consequently, the model $G/H$ of the Cartan geometry must be 
compact.}

\item \textit{Non-projective compact simply connected 
K\"ahler Calabi--Yau manifolds do not admit any branched holomorphic projective 
structure.}
\end{enumerate}

The structure of this paper is as follows. Section \ref{s2} introduces the 
main notation and definitions. Section \ref{s3} gives interesting examples of 
branched holomorphic Cartan geometries and contains the proofs of Proposition 
\ref{algebraic proj struct} and Proposition \ref{normal}. In Section \ref{s4} we 
give a criterion (Theorem \ref{thm1}) for the existence of a branched holomorphic 
Cartan geometry. In Section \ref{s5} we study holomorphic projective structures 
on compact parallelizable manifolds. Section \ref{Calabi--Yau} deals with branched 
holomorphic Cartan geometries on Calabi--Yau manifolds, and it contains the proofs of 
Theorem \ref{thm2} and Corollary \ref{final}.

\section{Holomorphic Cartan geometry and branched holomorphic Cartan geometry}\label{s2}

\subsection{Holomorphic Cartan geometry}\label{s2.1}

We first recall the definition of a holomorphic Cartan geometry.

Let $G$ be a connected complex Lie group and $H\, \subset\, G$ a connected
complex Lie subgroup. The Lie algebras of $H$ and $G$ will be denoted by
$\mathfrak h$ and $\mathfrak g$ respectively.

Let $X$ be a connected complex manifold and
\begin{equation}\label{e1}
f\,: E_H\, \longrightarrow\, X
\end{equation}
a holomorphic principal $H$--bundle on $X$. Let
\begin{equation}\label{e2}
E_G\, :=\, E_H\times^H G \, \stackrel{f_G}{\longrightarrow}\, X
\end{equation}
be the holomorphic principal $G$--bundle on $X$ obtained by extending the
structure group of $E_H$ using the inclusion of $H$ in $G$. So, $E_G$ is the
quotient of $E_H\times G$ where two points $(c_1,\, g_1),\, (c_2,\, g_2)\, \in\,
E_H\times G$ are identified if there is an element $h\, \in\, H$ such that
$c_2\,=\, c_1h$ and $g_2\,=\, h^{-1}g_1$. The projection $f_G$ in \eqref{e2} is induced
by the map $E_H\times G\, \longrightarrow\, X$, $(c,\, g)\,\longmapsto\, f(c)$, where
$f$ is the projection in \eqref{e1}. The action of $G$ on $E_G$ is induced by the
action of $G$ on $E_H\times G$ given by the right--translation action of $G$ on itself.
Let ${\rm ad}(E_H)\,=\, E_H\times^H {\mathfrak h}$ and
${\rm ad}(E_G)\,=\, E_G\times^G {\mathfrak g}$ be the adjoint vector bundles for
$E_H$ and $E_G$ respectively. We recall that ${\rm ad}(E_H)$ (respectively,
${\rm ad}(E_G)$) is the quotient of $E_H\times \mathfrak h$ (respectively, $E_G\times
\mathfrak g$) where two points $(z_1,\, v_1)$ and $(z_2,\, v_2)$ are identified if there
is an element $g\, \in\, H$ (respectively, $g\, \in\, G$) such that
$z_2\, =\, z_1g$ and $v_1$ is taken to $v_2$ by the automorphism of the Lie algebra
$\mathfrak h$ (respectively, $\mathfrak g$) given by automorphism of the Lie group
$H$ (respectively, $G)$ defined by $y\,\longmapsto\, g^{-1}yg$.
We have a short exact sequence of holomorphic vector bundles
on $X$
\begin{equation}\label{e3}
0\, \longrightarrow\, {\rm ad}(E_H)\, \stackrel{\iota_1}{\longrightarrow}\, {\rm ad}(E_G)
\, \longrightarrow\, {\rm ad}(E_G)/{\rm ad}(E_H)\, \longrightarrow\,0\, .
\end{equation}
The holomorphic tangent bundle of a complex manifold $Y$ will be denoted by $TY$. Let
$$
{\rm At}(E_H)\,=\, (TE_H)/H\, \longrightarrow\, X \ \text{ and } \ 
{\rm At}(E_G)\,=\, (TE_G)/G\, \longrightarrow\, X
$$
be the Atiyah bundles for $E_H$ and $E_G$ respectively; see \cite{At}. Let
\begin{equation}\label{e4}
0\, \longrightarrow\, {\rm ad}(E_H)\, \stackrel{\iota_2}{\longrightarrow}\, {\rm At}(E_H)
\, \stackrel{q_H}{\longrightarrow}\, TX\, \longrightarrow\,0
\end{equation}
and
\begin{equation}\label{e5}
0\, \longrightarrow\, {\rm ad}(E_G)\, \stackrel{\iota_0}{\longrightarrow}\, {\rm At}(E_G)
\, \stackrel{q_G}{\longrightarrow}\, TX\, \longrightarrow\,0
\end{equation}
be the Atiyah exact sequences for $E_H$ and $E_G$ respectively; see \cite{At}. The
projection $q_H$ (respectively, $q_G$) is induced by the differential of the map
$f$ (respectively, $f_G$) in \eqref{e1} (respectively,
\eqref{e2}). A holomorphic connection on a holomorphic principal bundle
is defined to be a holomorphic splitting of the Atiyah exact sequence associated to the
principal bundle \cite{At}. Therefore, a holomorphic connection on $E_G$ is a
holomorphic homomorphism
$$
\psi\, :\, {\rm At}(E_G)\, \longrightarrow\, {\rm ad}(E_G)
$$
such that $\psi\circ \iota_0\,=\, \text{Id}_{{\rm ad}(E_G)}$.

A holomorphic Cartan geometry on $X$ of type $G/H$ is a pair $(E_H,\, \theta)$, where
$E_H$ is a holomorphic principal $H$--bundle on $X$ and
$$
\theta\, :\, {\rm At}(E_H)\, \longrightarrow\, {\rm ad}(E_G)
$$
is a holomorphic isomorphism of vector bundles such that
$\theta\circ \iota_2 \,=\,\iota_1$ (see \eqref{e4} and \eqref{e3} for
$\iota_2$ and $\iota_1$ respectively). Therefore, we have 
the following commutative diagram
$$
\begin{matrix}
0 & \longrightarrow & {\rm ad}(E_H) & \stackrel{\iota_2}{\longrightarrow} & {\rm At}(E_H)
& \stackrel{q_H}{\longrightarrow} & TX & \longrightarrow & 0\\
&& \Vert && ~\Big\downarrow\theta && ~\Big\downarrow \phi\\
0 & \longrightarrow & {\rm ad}(E_H) & \stackrel{\iota_1}{\longrightarrow} & {\rm ad}(E_G)
& \longrightarrow & {\rm ad}(E_G)/{\rm ad}(E_H) & \longrightarrow & 0
\end{matrix}
$$
\cite[Ch.~5]{Sh}; the above homomorphism $\phi$ induced by
$\theta$ is evidently an isomorphism.

We can embed $\text{ad}(E_H)$ in ${\rm At}(E_H)\oplus {\rm ad}(E_G)$ by sending any 
$v$ to $(\iota_2(v),\, -\iota_1(v))$ (see \eqref{e4}, \eqref{e3}). The Atiyah bundle 
${\rm At}(E_G)$ is the quotient $({\rm At}(E_H)\oplus {\rm ad}(E_G))/\text{ad}(E_H)$ 
for this embedding. The inclusion of $\text{ad}(E_G)$ in ${\rm At}(E_G)$ in 
\eqref{e5} is given by the inclusion $\iota_1$ or $\iota_2$ of $\text{ad}(E_G)$ in
${\rm At}(E_H)\oplus {\rm ad}(E_G)$ (note that they give the same homomorphism to
the quotient bundle $({\rm At}(E_H)\oplus {\rm ad}(E_G))/\text{ad}(E_H)$).

Given a holomorphic Cartan geometry $(E_H,\, \theta)$ of type $G/H$ on $X$, the homomorphism
$$
{\rm At}(E_H)\oplus {\rm ad}(E_G)\, \longrightarrow\, {\rm ad}(E_G),\, \ \
(v,\, w) \, \longmapsto\, \theta(v)+w
$$
produces a homomorphism
$$
\theta'\, :\, {\rm At}(E_G)\,=\,
({\rm At}(E_H)\oplus {\rm ad}(E_G))/\text{ad}(E_H)\, \longrightarrow\, {\rm ad}(E_G)
$$
which satisfies the condition that $\theta'\circ \iota_0\,=\, \text{Id}_{{\rm 
ad}(E_G)}$, meaning $\theta'$ is a holomorphic splitting of \eqref{e5}. Therefore, 
$\theta'$ is a holomorphic connection on the principal $G$--bundle $E_G$.

The curvature $\text{Curv}(\theta')$ of the connection $\theta'$ is a holomorphic section
$$
\text{Curv}(\theta')\, \in\, H^0(X,\, \text{ad}(E_G)\otimes \Omega^2_X)\, ,
$$
where $\Omega^i_X\, :=\, \bigwedge\nolimits^i (TX)^*$.

The Cartan geometry $(E_H,\, \theta)$ is called \textit{normal} if
$$
\text{Curv}(\theta')\, \in\, H^0(X,\, \text{ad}(E_H)\otimes \Omega^2_X)
$$
\cite[Ch.~8, \S~2, p.~338]{Sh}.

The Cartan geometry $(E_H,\, \theta)$ is called \textit{flat} if
$$
\text{Curv}(\theta')\,=\,0
$$
\cite[Ch.~5, \S~1, p.~177]{Sh}. Consequently, flat Cartan geometries are normal.

If $(E_H,\, \theta)$ is a holomorphic Cartan geometry, then the isomorphism $\theta$ can be
interpreted as a $\mathfrak g$--valued holomorphic $1$--form $\beta$ on $E_H$ satisfying the
following three conditions:
\begin{enumerate}
\item the homomorphism $\beta\, :\, TE_H\, \longrightarrow\, E_H\times{\mathfrak g}$ is an
isomorphism,

\item $\beta$ is $H$--equivariant with $H$ acting on $\mathfrak g$ via conjugation, and

\item the restriction of $\beta$ to each fiber of $f$ coincides with the Maurer--Cartan form
associated to the action of $H$ on $E_H$.
\end{enumerate}
(See \cite{Sh}.)

\subsection{Developing curves}

This subsection is based on \cite{Br}.
Consider $ f_G\,=\, E_G \,\longrightarrow\, X$ and a second projection map $h_G\,:\, E_G
\,\longrightarrow\, G/H$ 
which associates to each class $(c,\,g) \,\in\, E_G$ the element $gH \,\in\, G/H$.

The differentials of the projections $f_G$ and $h_G$ map the horizontal space of the connection 
$\theta'$ at $(c,\,g) \,\in\, E_G$ isomorphically onto $T_{f_G(c)} X$ and onto $T_{gH} (G/H)$ 
respectively. This defines an isomorphism between $T_{f_G(c)} X$ and $T_{gH} (G/H)$. Hence a 
Cartan geometry provides a family of $1$-jets identifications of the manifold $X$ with the 
model space $G/H$.

Also, the connection $\theta'$ defines a way to lift differentiable curves in $X$ to 
$\theta'$-horizontal curves in $E_G$. Moreover this lift is unique once one specifies the 
starting point of the lifted curve.

\subsection{Branched holomorphic Cartan geometry}\label{s2.2}

\begin{definition}\label{def1}
A \textit{branched} holomorphic Cartan geometry on $X$ of type $G/H$ is a pair $(E_H,\, \theta)$,
where $E_H$ is a holomorphic principal $H$--bundle on $X$ and
$$
\theta\, :\, {\rm At}(E_H)\, \longrightarrow\, {\rm ad}(E_G)
$$
is a holomorphic homomorphism of vector bundles, such that following three conditions hold:
\begin{enumerate}
\item $\theta$ is an isomorphism over a nonempty open subset of $X$, and

\item $\theta\circ \iota_2 \,=\,\iota_1$ (see \eqref{e4} and \eqref{e3}).
\end{enumerate}
\end{definition}

In other words, we have a commutative diagram
\begin{equation}\label{e6}
\begin{matrix}
0 & \longrightarrow & {\rm ad}(E_H) & \stackrel{\iota_2}{\longrightarrow} & {\rm At}(E_H)
& \stackrel{q_H}{\longrightarrow} & TX & \longrightarrow & 0\\
&& \Vert && ~\Big\downarrow\theta && ~\Big\downarrow \phi\\
0 & \longrightarrow & {\rm ad}(E_H) & \stackrel{\iota_1}{\longrightarrow} & {\rm ad}(E_G)
& \longrightarrow & {\rm ad}(E_G)/{\rm ad}(E_H) & \longrightarrow & 0
\end{matrix}
\end{equation}
of holomorphic vector bundles on $X$, such that
$\theta$ is an isomorphism over a nonempty open subset of $X$;
the homomorphism $\phi$ in \eqref{e6} is induced by $\theta$.

Let $U\, \subset\, X$ be the nonempty open subset over which $\theta$ is an 
isomorphism. From the commutativity of \eqref{e6} it follows that $\phi$ is an 
isomorphism exactly over $U$.

\begin{lemma}\label{lem1}
The complement $X\setminus U$ is a divisor.
\end{lemma}

\begin{proof}
Let $d$ be the complex dimension of $X$. The homomorphism $\phi$ in \eqref{e6} produces
a homomorphism
\begin{equation}\label{f1}
\bigwedge\nolimits^d\phi\, :\, \bigwedge\nolimits^d TX\, \longrightarrow\, \bigwedge\nolimits^d
({\rm ad}(E_G)/{\rm ad}(E_H))\, ,
\end{equation}
so $\bigwedge^d\phi$ is a holomorphic section of the line bundle
$(\bigwedge^d ({\rm ad}(E_G)/{\rm ad}(E_H)))\otimes \Omega^d_X$. The homomorphism
$\phi$ is an isomorphism exactly on the complement of the 
divisor associated
to this section $\bigwedge^d\phi$ of $(\bigwedge^d ({\rm ad}(E_G)/{\rm
ad}(E_H)))\otimes \Omega^d_X$. Since $\phi$ is an
isomorphism exactly over $U$, the complement $X\setminus U$ coincides with the support
of the divisor associated to the above section $\bigwedge^d\phi$.
\end{proof}

\begin{definition}\label{def2}
The divisor associated to the above section $\bigwedge^d\phi$ of
$(\bigwedge^d ({\rm ad}(E_G)/{\rm ad}(E_H)))\otimes \Omega^d_X$ will be called the
\textit{branching divisor} for 
the branched holomorphic Cartan geometry $(E_H,\, \theta)$ on $X$.
\end{definition}

As in the case of usual Cartan geometries, consider the homomorphism
$$
{\rm At}(E_H)\oplus {\rm ad}(E_G)\, \longrightarrow\, {\rm ad}(E_G),\, \ \
(v,\, w) \, \longmapsto\, \theta(v)+w\, .
$$
Since ${\rm At}(E_G)\,=\, ({\rm At}(E_H)\oplus {\rm ad}(E_G))/\text{ad}(E_H)$,
the above homomorphism produces a holomorphic connection
\begin{equation}\label{thp}
\theta'\, :\, {\rm At}(E_G)\, \longrightarrow\, {\rm ad}(E_G)
\end{equation}
on $E_G$.

We will call a branched holomorphic Cartan geometry $(E_H,\, \theta)$ to be \textit{normal}
if
$$
\text{Curv}(\theta')\, \in\, H^0(X,\, \text{ad}(E_H)\otimes \Omega^2_X)\, .
$$

We will call a branched holomorphic Cartan geometry $(E_H,\, \theta)$ to be \textit{flat}
if
$$
\text{Curv}(\theta')\,=\,0\, .
$$

If $(E_H,\, \theta)$ is a branched holomorphic Cartan geometry, then the homomorphism $\theta$ can be
interpreted as a $\mathfrak g$--valued holomorphic $1$--form $\beta$ on $E_H$ satisfying the
following three conditions:
\begin{enumerate}
\item the homomorphism $\beta\, :\, TE_H\, \longrightarrow\, E_H\times{\mathfrak g}$ is an
isomorphism over a nonempty open subset of $E_H$,

\item $\beta$ is $H$--equivariant with $H$ acting on $\mathfrak g$ via conjugation, and

\item the restriction of $\beta$ to each fiber of $f$ coincides with the Maurer--Cartan form
associated to the action of $H$ on $E_H$.
\end{enumerate}

\subsection{The developing map}\label{sdm}

Let $X$ be a simply connected complex manifold (it need not be compact), and let $E_H$
be a holomorphic principal $H$--bundle on $X$. Consider the holomorphic principal $G$--bundle
$E_G\,:=\, E_H\times^H G\,\longrightarrow\, X$. Let
$$
\theta\, :\, {\rm At}(E_H)\, \longrightarrow\, {\rm ad}(E_G)
$$
be a branched holomorphic Cartan geometry on $X$ of type $G/H$ such that the
associated connection $\theta'$ on $E_G$ (constructed in \eqref{thp}) is flat.

Since $E_G$ is a flat principal $G$--bundle over a simply connected manifold, we conclude
that $E_G$ is trivializable using the flat connection. Note that the trivialization
of $E_G$ is not quite unique. Once we fix a point $z_0$ of $E_G$, there is a unique
isomorphism $\xi_{z_0}$ of $E_G$ with the trivial principal $G$--bundle $X\times G
\,\longrightarrow\, X$ satisfying the
following two conditions:
\begin{enumerate}
\item $\xi_{z_0}$ takes the connection $\theta'$ on $E_G$ to the trivial connection
on the trivial principal $G$--bundle $X\times G$, and

\item $\xi_{z_0}(z_0) \,=\, (f_G(z_0),\, e)$, where $f_G$
is the projection (as in \eqref{e2}) of $E_G$ to $X$, and $e\, \in\, G$ is the identity element.
\end{enumerate}
If we replace $z_0$ by another point $z'_0$ of $E_G$, then there is a unique element $g$
of $G$ such that the following diagram is commutative:
$$
\begin{matrix}
E_G & = & E_G\\
~~\,~~\,~~\, \Big\downarrow\xi_{z_0} && ~~\,~~\,~~\, \Big\downarrow\xi_{z'_0}\\
X\times G &\stackrel{L_g}{\longrightarrow} & X\times G
\end{matrix}
$$
where $L_g$ is the diffeomorphism $(x,\, h)\, \longmapsto\, (x,\, gh)$ of $X\times G$.

Fix an element $z_0\,\in\, E_G$. Identify $E_G$ with the trivial principal $G$--bundle $X\times G
\,\longrightarrow\, X$ using $z_0$. Using this identification, the reduction $E_H$ of $E_G$
becomes a holomorphic reduction of structure group of $X\times G\,\longrightarrow\,
X$ to the subgroup $H\, \subset\, G$. Now we observe that any holomorphic reduction of $X\times G
\,\longrightarrow\, X$ to $H$ is given by a holomorphic map
$$\varpi \, :\, X\, \longrightarrow\, G/H\, .$$
To see this, let $p_H\, :\, G\, \longrightarrow\, G/H$ be the quotient map. Then
$$(\text{Id}_X\times p_H)^{-1}(\text{graph}(\varpi))\, \subset\, X\times G$$ is a reduction
of structure group to $H$, where $\text{graph}(\varpi)\, \subset\, X\times (G/H)$ is the
graph of $\varpi$ consisting of all points of the form $\{(x,\, \varpi(x))\}_{x\in X}$. Conversely,
any holomorphic reduction of structure group of $X\times G\,\longrightarrow\,
X$ to $H$ gives a holomorphic map from $X$ to $G/H$.

Let
$$
\varpi\, :\, X\, \longrightarrow\, G/H
$$
be the holomorphic map for the reduction $E_H$ of $E_G\,=\, X\times G$ to $H$. This map $\varpi$
is a developing map for the branched holomorphic Cartan geometry $\theta$. If we replace
$z_0$ by another point of $E_G$, then the developing map gets composed with a left--translation
of $G/H$ by an element of $G$.

{}From the definition of a branched holomorphic Cartan geometry of type $G/H$ it follows that
$\varpi$ is a local biholomorphism over the complement of the branching divisor.

\section{Examples of branched holomorphic Cartan geometries}\label{s3}

\subsection{The standard model}

We recall the standard (flat) Cartan geometry of type $G/H$.

Set $X\,=\, G/H$. Let $F_H$ be the holomorphic principal $H$--bundle on $X$ defined by the quotient
map $G\, \longrightarrow\, G/H$; we use the notation $F_H$ instead of $E_H$ because
it is a special case which will play a role later. Identify the Lie algebra $\mathfrak g$ with the
right--invariant vector fields on $G$. This produces an isomorphism of $\text{At}(F_H)$ with
$\text{ad}(F_G)$ and hence a Cartan geometry of type $G/H$ on $X$ (we use the notation
$F_G$ instead of $E_G$ for the same reason as above). Equivalently, the tautological holomorphic
$\mathfrak g$--valued $1$--form on $G\,=\, F_H$ satisfies all the three conditions needed
to define a Cartan geometry of type $G/H$ (see the last paragraph of Section \ref{s2.1}).

The above holomorphic $\mathfrak g$--valued $1$--form on $G\,=\, F_H$ will be denoted by
$\theta_{G,H}$.

\subsection{Flat Cartan geometries}

A $(E_H,\, \theta)$ holomorphic Cartan geometry of type $G/H$ is flat if and only if it is 
locally isomorphic to $(F_H,\,\theta_{G,H})$ \cite[Ch.~5, \S~5, Theorem 5.1]{Sh}.

If a complex manifold $X$ admits a flat holomorphic Cartan geometry of type $G/H$, then $X$ 
admits a covering by open subsets $U_i$ and biholomorphisms onto open subsets
of $G/H$
$$\phi_i \,:\, U_i 
\,\longrightarrow\, G/H$$ such that each transition map $$\phi_{i} \circ \phi_{j}^{-1} \,:\, 
\phi_{j}(U_i \cap U_j)\,\longrightarrow\, \phi_{i }(U_i \cap U_j)$$ is, on each connected 
component, the restriction of an automorphism of $G/H$ given by the left--translation action of
an element $g_{ij} \,\in\, G$ \cite[Ch.~5, \S~5, Theorem 5.2]{Sh}. 

Following Ehresmann, \cite{Eh}, one classically defines then a {\it monodromy 
homomorphism} $\rho$ from the fundamental group $\pi_1(X)$ of $X$ into $G$ and a 
developing map $\delta \,:\, \widetilde{X} \,\longrightarrow\, G/H$ which is a 
$\pi_1(X)$-equivariant local biholomorphism from the universal cover $\widetilde X$ 
into the model $G/H$.

The same strategy was used in Section \ref{sdm} to adapt the proof to the branched case.

\subsection{Construction of branched holomorphic Cartan geometries}\label{secbc}

Let $X$ be a connected complex manifold and
$$
\gamma\, :\, X\, \longrightarrow\, G/H
$$
a holomorphic map such that the differential
$$
d\gamma\, :\, TX\, \longrightarrow\, T(G/H)
$$
is an isomorphism over a nonempty subset of $X$.

The above condition on $d\gamma$ is equivalent to the condition that $\dim X\,=\, \dim (G/H)$ 
with $\gamma(X)$ containing a nonempty open subset of $G/H$. Note that in general, the 
homomorphism $d\gamma$ is always an isomorphism over an open subset of $X$, which may be empty.

Set $E_H$ to be the pullback $\gamma^*F_H$. Note that we have a holomorphic
map $\eta\, :\, E_H\, \longrightarrow\, F_H$ which is $H$--equivariant and
fits in the commutative diagram
$$
\begin{matrix}
E_H & \stackrel{\eta}{\longrightarrow} & F_H\\
\Big\downarrow && \Big\downarrow\\
X & \stackrel{\gamma}{\longrightarrow} & G/H
\end{matrix}
$$
Then $(E_H,\, \eta^*\theta_{G,H})$ defines a branched Cartan geometry of type $G/H$ on $X$.

To describe the above branched Cartan geometry in terms of the Atiyah bundle, first note
that $\text{At}(\gamma^* F_H)$ coincides with the subbundle of the vector bundle
$\gamma^*\text{At}(F_H) \oplus TX$ given by the kernel of the homomorphism
$$
\gamma^*\text{At}(F_H)\oplus TX\, \longrightarrow\, \gamma^*T(G/H)\, ,\ \
(v,\, w)\, \longmapsto\, \gamma^*q_{G,H}(v) -d\gamma(w)\, ,
$$
where $q_{G,H}\, :\, \text{At}(F_H)\, \longrightarrow\, T(G/H)$ is the natural projection
(see \eqref{e4}), and $$d\gamma\,:\, TX\,\longrightarrow\, \gamma^* T(G/H)$$ is the
differential of $\gamma$. Consider the standard Cartan geometry $\theta_{G,H}\, :\,
\text{At}(F_H)\, \stackrel{\sim}{\longrightarrow}\, \text{ad}(F_G)$ of type $G/H$ on
the quotient $G/H$. The restriction of the homomorphism
$$
\gamma^*\text{At}(F_H)\oplus TX\, \longrightarrow\, \gamma^*\text{ad}(F_G)\, ,
\ \ (a,\, b)\, \longmapsto\, \gamma^*\theta_{G,H} (a)
$$
to $\text{At}(\gamma^*F_H)\, \subset\, \gamma^*\text{At}(F_H)\oplus TX$ is a homomorphism
$$
\text{At}(\gamma^*F_H)\, \longrightarrow\, \text{ad}(\gamma^*F_G)\,=\, \gamma^*\text{ad}(F_G)
\,=\, \text{ad}(E_G)\, ,
$$
which defines a branched holomorphic Cartan geometry of type $G/H$ on $X$.

The divisor of $X$ over which the above branched Cartan geometry of type $G/H$ on $X$ fails to 
be a Cartan geometry evidently coincides with the divisor over which the
differential $d\gamma$ fails to be an isomorphism.

The curvature of the holomorphic connection on $E_G$ associated to the above branched 
Cartan geometry of type $G/H$ on $X$ vanishes identically. Indeed, this
follows immediately from the fact that the standard Cartan geometry $\theta_{G,H}$ is flat.
In particular, this branched Cartan geometry on $X$ is normal.

The developing map for this flat branched Cartan geometry on $X$ is the map $\gamma$ itself.

Conversely, let $X$ be a complex manifold endowed with a branched flat holomorphic Cartan 
geometry with branching divisor $D$. Then Section \ref{sdm} (which is an adaptation in the branched case of
the proof of Theorem 5.2 in \cite[Ch.~5, \S~5]{Sh})
shows that $X$ admits a covering by open (simply connected) subsets $U_i$ such that for each $i$ there exists
a holomorphic map $\phi_i \,:\, 
U_i \,\longrightarrow\, G/H$ which is a local biholomorphism on
$U_i \setminus (U_i \cap D)$; 
moreover, for every pair $i,\, j$, each connected component of the overlap $U_i \cap U_j$, there
exists a $g_{ij}\,\in\, G$ such that $g_{ij} \circ f_j\,=\,f_i$ on the entire connected component.

Then the Ehresmann method \cite{Eh} (based by analytic continuation of charts along paths) 
defines a monodromy morphism $\rho \,:\, \pi_1(X)\,\longrightarrow\, G$ and a developing map 
$\delta \,:\, \widetilde{X} \,\longrightarrow\, G/H$ which is a local biholomorphism away from 
the pull-back of $D$ to the universal cover.

\subsection{Branched flat affine and projective structures}

Let us recall the standard model G/H of the affine geometry.

Consider the semi-direct product ${\mathbb C}^d\rtimes\text{GL}(d, {\mathbb C})$ for the standard 
action of $\text{GL}(d, {\mathbb C})$ on ${\mathbb C}^d$. This group ${\mathbb 
C}^d\rtimes\text{GL}(d, {\mathbb C})$ is identified with the group of all affine transformations 
of ${\mathbb C}^d$. Set $H\,=\, \text{GL}(d, {\mathbb C})$ and $G\,=\, {\mathbb 
C}^d\rtimes\text{GL}(d, {\mathbb C})$.

A {\it holomorphic affine structure} (or equivalently {\it holomorphic affine connection}) on a 
complex manifold $X$ of dimension $d$ is a holomorphic Cartan geometry of type $G/H$. This
terminology comes from the fact that the bundle $F_H$ will be automatically isomorphic to the 
holomorphic frame bundle of $X$ and the form $\theta_{G,H}$ defines a holomorphic connection in 
the holomorphic tangent bundle of $X$. Conversely, any holomorphic connection in the holomorphic 
tangent bundle of $X$ uniquely defines a holomorphic Cartan geometry of type $G/H$ (where
$G$ and $H$ are as above). This connection 
is torsionfree exactly when the Cartan geometry is normal \cite{MM}. For more details on the 
equivalence between the several definitions of a holomorphic affine connection (especially with 
the one seeing the connection as an operator $\nabla$ acting on local holomorphic vector fields 
and satisfying the Leibniz rule), the reader is referred to \cite{MM,Sh}.

A branched holomorphic Cartan geometry of type ${\mathbb C}^d\rtimes\text{GL}(d, {\mathbb C}) 
/\text{GL}(d, {\mathbb C})$ will be called a {\it branched holomorphic affine structure} or a 
{\it branched holomorphic affine connection}.

We also recall that a {\it holomorphic projective structure} (or a {\it holomorphic projective connection}) on a complex manifold $X$ of dimension $d$ is a 
holomorphic Cartan geometry of type $\text{PGL}(d+1,{\mathbb C})/Q$, where $Q\, \subset\, 
\text{PGL}(d+1,{\mathbb C})$ is the maximal parabolic subgroup that fixes a given point for the 
standard action of $\text{PGL}(d+1,{\mathbb C})$ on ${\mathbb C}P^d$ (the space of lines in 
${\mathbb C}^{d+1}$). In particular, there is a standard holomorphic projective structure on 
$\text{PGL}(d+1,{\mathbb C})/Q\,=\, {\mathbb C}P^d$. Locally a holomorphic projective connection is an equivalence class of holomorphic affine connections, where
two affine connections are considered to be equivalent if they admit the same unparametrized geodesics. The projective connection is normal exactly when it admits a local representative which is a torsionfree
affine connection \cite{MM,OT}.

We will call a branched holomorphic Cartan 
geometry of type $\text{PGL}(d+1,{\mathbb C})/Q$ a {\it branched holomorphic projective structure} or a {\it branched holomorphic projective connection.}

\begin{proposition} \label {algebraic proj struct}
Every compact complex projective manifold admits a branched flat holomorphic
projective structure.
\end{proposition}

\begin{proof}
Let $X$ be a compact complex projective manifold of complex dimension $d$. Then there exists a finite
surjective algebraic, hence holomorphic, morphism
$$
\gamma\, :\, X\, \longrightarrow\, {\mathbb C}P^d\, . 
$$

Indeed, one proves that the smallest integer $N$ for which there exists a finite 
morphism $f$ from $X$ to ${\mathbb C}P^N$ is $d$. If $N \,>\,d$, then there exists $P 
\,\in\, {\mathbb C}P^N \setminus f(X)$; now consider the projection $\pi \,:\, {\mathbb 
C}P^N \setminus \{P \} \,\longrightarrow\, {\mathbb C}P^{N-1}$. The fibers of $\pi \circ f$ must be 
finite (otherwise $f(X)$ would contain a line through $P$, hence $P$). Since $\pi 
\circ f$ is a proper morphism with finite fibers, it must be finite.

Now we can pull back the standard holomorphic projective structure on ${\mathbb C}P^d$
using the above map $\gamma$ to get a branched holomorphic projective structure on $X$.
\end{proof}

\begin{proposition}\label{simply connected}\mbox{}
\begin{enumerate}
\item[(i)] Simply connected compact complex manifolds do not admit any branched flat
holomorphic affine structure.

\item[(ii)] Simply connected compact complex manifolds admitting a branched flat holomorphic
projective structure are Moishezon.
\end{enumerate}
\end{proposition}

\begin{proof} 
(i) If, by contradiction, a simply connected compact complex manifold $X$ admits a branched flat 
holomorphic affine structure, then the developing map $\delta \,:\, X \,\longrightarrow\, 
{\mathbb C}^d$ is holomorphic and nonconstant, which is a contradiction.

(ii) If $X$ is a simply connected manifold of complex dimension $d$ admitting a branched flat
holomorphic projective structure, then its developing map is a holomorphic map $\delta \,:\, X \,
\longrightarrow\, {\mathbb C}P^d$ which is a local biholomorphism away from a divisor $D$ in
$X$. Thus, the algebraic dimension of $X$ must be $d$; consequently, $X$ is Moishezon.
\end{proof}

Since any given compact K\"ahler manifold is Moishezon if and only if it is projective, one gets
the following:

\begin{corollary}\label{corollary simply connected}
Non-projective simply connected K\"ahler manifolds do not admit any branched flat holomorphic
projective structure.
\end{corollary}

In particular, non-projective $K3$ surfaces do not admit any branched flat holomorphic projective 
structure.

\subsection{Branched normal holomorphic projective structure on complex surfaces}

In \cite{KO}, Kobayashi and Ochiai classified all compact complex surfaces admitting a holomorphic 
projective structure (connection). All of them happen to be isomorphic to quotient of open 
subsets of ${\mathbb C}P^2$ by discrete subgroups of $\text{PGL}(3, {\mathbb C})$ acting properly 
and discontinuously. Consequently, all of them also admit a flat holomorphic projective structure. 
Among those surfaces, the only projective ones are the following : ${\mathbb C}P^2$, surfaces 
covered by the ball and the abelian varieties (and their finite unramified quotients).

Moreover, it is known that every normal projective structure (connection) on a compact complex 
surface is automatically flat \cite{D2}.

Proposition \ref{algebraic proj struct} shows that the class of compact complex surfaces 
admitting a branched holomorphic projective structure is much broader. Moreover, we have the 
following:

\begin{proposition}\label{normal}
There exists branched holomorphic projective structures on compact complex surfaces which
are normal, but not flat.
\end{proposition}

\begin{proof}
Let $Y$ be a compact connected Riemann surface of genus at least two. Fix two holomorphic
$1$--forms $\alpha_1,\, \alpha_2\, \in\, H^0(X,\, \Omega^1_X)$ that are linearly
independent. Set $X\,=\, Y\times Y$.
Let $Q\, \subset\, \text{PGL}(3, {\mathbb C})$ be the
maximal parabolic subgroup that fixes the point $(1,\, 0, \, 0)\, \in\, {\mathbb C}P^2$
for the standard action of $\text{PGL}(3, {\mathbb C})$ on ${\mathbb C}P^2$. Set
$H\,=\, Q$ and $G\,=\, \text{PGL}(3, {\mathbb C})$.

Let $E_H\,=\, X\times H\,\stackrel{f}{\longrightarrow}\, X$ be the trivial holomorphic principal 
$H$--bundle on $X$. So, the corresponding holomorphic principal $G$--bundle $E_G$ is the trivial 
holomorphic principal $G$--bundle $X\times G\,\longrightarrow\, X$. The adjoint vector bundles 
$\text{ad}(E_H)$ and $\text{ad}(E_G)$ are the trivial vector bundles $X\times\mathfrak h\, 
\longrightarrow\, X$ and $X\times\mathfrak g\,\longrightarrow\, X$ respectively. The 
trivialization of $E_H$ produces a trivial holomorphic connection on $E_H$. This connection 
defines a holomorphic splitting of the Atiyah exact sequence in \eqref{e4}. Hence we have
$$
\text{At}(E_H)\,=\, \text{ad}(E_G)\oplus TX\,=\, (X\times{\mathfrak h})\oplus TX\, .
$$
Now let $$\theta\, :\, \text{At}(E_H)\,\longrightarrow\, \text{ad}(E_G)\,=\,X\times\mathfrak g$$
be the holomorphic homomorphism which over any point $(y_1,\, y_2)\,\in\, Y\times Y\,=\, X$
is defined by 
$$
(w,\, (v_1,\, v_2))\, \longmapsto\, w +
\begin{pmatrix}
0 & 0 & \alpha_1(y_1)(v_1)\\
\alpha_1(y_1)(v_1) & 0 &0\\
\alpha_2(y_2)(v_2) & 0 &0
\end{pmatrix}\, , \ w\,\in\, {\mathfrak h}\, ,\ v_i\, \in\, T_{y_i}Y\, .
$$
Note that the Lie algebra $\mathfrak g$ is the space of $3\times 3$ complex matrices
of trace zero, while $\mathfrak h$ is the subalgebra of $\mathfrak g$ consisting of
matrices $(a_{i,j})_{i,j=1}^3$ with complex entries
such that $a_{2,1}\,=\, 0\,=\, a_{3,1}$. Therefore,
$\theta$ is an isomorphism over the nonempty open subset of $X$ consisting of all
$(y_1,\, y_2)\, \in\, Y\times Y$ such that both $\alpha_1(y_1)$ and $\alpha_2(y_2)$ are
nonzero.

Let $\theta'$ be the holomorphic connection on $E_G\,=\, X\times G\,\longrightarrow\, X$ associated to
$\theta$ (see \eqref{thp}). To describe $\theta'$, let $D_0$ denote the trivial
holomorphic connection on $E_G\,=\, X\times G$ given by its trivialization. Let 
$$
p_i\, :\, X\,=\, Y\times Y\,\longrightarrow\, Y\, , \ \ i\,=\, 1,\, 2
$$
be the projection to the $i$--th factor. Then we have
$$
\theta'\,=\, D_0+ \begin{pmatrix}
0 & 0 & p^*_1\alpha_1\\
p^*_1\alpha_1 & 0 &0\\
p^*_2\alpha_2 & 0 &0
\end{pmatrix}\, ;
$$
note that $\text{ad}(E_G)\,=\, X\times {\mathfrak g}$, and 
$$\begin{pmatrix}
0 & 0 & p^*_1\alpha_1\\
p^*_1\alpha_1 & 0 &0\\
p^*_2\alpha_2 & 0 &0
\end{pmatrix} \, \in\, H^0(X,\, \text{ad}(E_G)\otimes \Omega^1_X)
$$
because the diagonal entries are zero. Therefore, the curvature 
$\text{Curv}(\theta')$ of the connection $\theta'$ has the following
expression:
$$
\text{Curv}(\theta')\,=\,
\begin{pmatrix}
0 & 0 & p^*_1\alpha_1\\
p^*_1\alpha_1 & 0 &0\\
p^*_2\alpha_2 & 0 &0
\end{pmatrix} 
\wedge 
\begin{pmatrix}
0 & 0 & p^*_1\alpha_1\\
p^*_1\alpha_1 & 0 &0\\
p^*_2\alpha_2 & 0 &0
\end{pmatrix} 
\,=\,
\begin{pmatrix}
(p^*_1\alpha_1)\wedge (p^*_2\alpha_2) & 0 & 0\\
0 & 0 &0\\
0& 0& (p^*_2\alpha_2)\wedge (p^*_1\alpha_1)
\end{pmatrix} 
$$
Hence we have
$$
\text{Curv}(\theta')\, \in\, H^0(X,\, \text{ad}(E_H)\otimes \Omega^2_X)\, .
$$
So the branched projective structure $(E_H,\, \theta)$ constructed above in
normal. But we have $\text{Curv}(\theta')\,\not=\, 0$.
\end{proof}

We don't know whether (non-projective) compact complex surfaces admitting branched holomorphic 
projective structures are exactly those admitting a branched flat holomorphic projective 
structure.

\section{A criterion}\label{s4}

Let $X$ be a compact connected K\"ahler manifold of complex dimension $d$ equipped with a K\"ahler
form $\omega$. Chern classes will always mean ones with real coefficients. For a torsionfree
coherent analytic sheaf $V$ on $X$, define
\begin{equation}\label{deg}
\text{degree}(V)\,:=\, (c_1(V)\cup\omega^{d-1})\cap [X]\, \in\, {\mathbb R}\, .
\end{equation}
The degree of a divisor $D$ on $X$ is defined to be $\text{degree}({\mathcal O}_X(D))$.
The degree of a general coherent analytic sheaf on $X$ is the degree of its torsionfree
quotient.

Fix an effective divisor $D$ on $X$. Fix a holomorphic principal $H$--bundle $E_H$ on $X$.

\begin{proposition}\label{thm1}
If ${\rm degree}(\Omega^1_X)-{\rm degree}(D)\, \not=\, {\rm degree}({\rm ad}(E_H))$, then there 
is no branched holomorphic Cartan geometry of type $G/H$ on $X$ with branching divisor $D$ (see 
Definition \ref{def2}). In particular, if $D\, \not=\, 0$ and ${\rm degree}(\Omega^1_X)\,\leq\, 
{\rm degree}({\rm ad}(E_H))$, then there is no branched holomorphic Cartan geometry of type 
$G/H$ on $X$ with branching divisor $D$.
\end{proposition}

\begin{proof}
Let $(E_H,\, \theta)$ be a branched holomorphic Cartan geometry of type $G/H$ on $X$ with 
branching divisor $D$. Consider the homomorphism $\bigwedge^d\phi$ in \eqref{f1}. Since $D$ is the 
divisor for the corresponding holomorphic section of the line bundle $(\bigwedge^d ({\rm 
ad}(E_G)/{\rm ad}(E_H)))\otimes \Omega^d_X$, we have
$$
\text{degree}(D)\,=\, \text{degree}((\bigwedge\nolimits^d ({\rm ad}(E_G)/{\rm ad}(E_H)))
\otimes \Omega^d_X)
$$
\begin{equation}\label{f2}
=\, \text{degree}({\rm ad}(E_G)) - \text{degree}({\rm ad}(E_H)) 
+ {\rm degree}(\Omega^1_X)\, .
\end{equation}
Recall that $E_G$ has a holomorphic connection $\theta'$ corresponding to $\theta$.
It induces a holomorphic connection on $\text{ad}(E_G)$. Hence we have
$c_1({\rm ad}(E_G)) \,=\, 0$ \cite[Theorem~4]{At}, which implies that
$\text{degree}({\rm ad}(E_G)) \,=\, 0$. Therefore, from \eqref{f2} it follows that
\begin{equation}\label{e7}
{\rm degree}(\Omega^1_X)-{\rm degree}(D)\, =\, {\rm degree}({\rm ad}(E_H))\, .
\end{equation}
Hence there
is no branched holomorphic Cartan geometry of type $G/H$ on $X$ with branching divisor $D$
if we have ${\rm degree}(\Omega^1_X)-{\rm degree}(D)\, \not=\, {\rm degree}({\rm ad}(E_H))$.

If $D\,\not=\, 0$, then $\text{degree}(D)\, >\, 0$. Hence in that case \eqref{e7} fails if
we have ${\rm degree}(\Omega^1_X)\,\leq\, {\rm degree}({\rm ad}(E_H))$.
\end{proof}

\begin{corollary}\label{corollaire deg}\mbox{}
\begin{enumerate}
\item[(i)] If ${\rm degree}(\Omega^1_X)\, <\, 0$, then there is no branched holomorphic affine structure
on $X$.

\item[(ii)] If ${\rm degree}(\Omega^1_X)\, =\, 0$, then all branched holomorphic affine structures
on $X$ are actually holomorphic affine structures.
\end{enumerate}
\end{corollary}

\begin{proof}
Set $H\,=\,
\text{GL}(d, {\mathbb C})$ and $G\,=\, {\mathbb C}^d\rtimes\text{GL}(d, {\mathbb C})$. Recall
that a branched holomorphic affine structure on $X$ is a 
branched holomorphic Cartan geometry on $X$ of type $G/H$, where $H$ and $G$ are as above.
Let $(E_H,\, \theta)$ be a branched holomorphic affine structure on the compact K\"ahler manifold
$(X,\, \omega)$ of dimension $d$. The homomorphism $$\text{M}(d, {\mathbb C})\otimes
\text{M}(d, {\mathbb C})\, \longrightarrow\, \mathbb C\, ,\ \ A\otimes B\, \longmapsto\,
\text{Trace}(AB)$$ is nondegenerate and $\text{GL}(d, {\mathbb C})$--invariant. In other words, 
the Lie algebra $\mathfrak h$ of $H\,=\, \text{GL}(d, {\mathbb C})$
is self-dual as an $H$--module. Hence we have $\text{ad}(E_H)\,=\, \text{ad}(E_H)^*$, in particular,
the equality $$\text{degree}(\text{ad}(E_H))\,=\,0$$ holds.

As noted before, for a nonzero effective divisor $D$ we have $\text{degree}(D)\, >\, 0$.
Therefore, the corollary follows from Proposition \ref{thm1}.
\end{proof}

\begin{remark}
Let $X$ be a rationally connected compact complex manifold. The proof of Theorem 4.1 in \cite{BM} 
extends to branched Cartan geometries on $X$. In other words, any branched Cartan geometry of 
type $G/H$ on $X$ is flat and it is given by a holomorphic map $X\, \longrightarrow\, G/H$ (see 
Section \ref{secbc}). This implies that $G/H$ is compact.
\end{remark}

\section{Holomorphic projective structure on parallelizable manifolds}\label{s5}

A complex manifold is called \textit{parallelizable} if its holomorphic tangent bundle
is holomorphically trivial.
We recall that, by a theorem of Wang \cite{Wa}, compact complex parallelizable 
manifolds are 
isomorphic to quotients $G /\Gamma$ of complex Lie groups $G$ by a {\it cocompact} 
lattice $\Gamma\,\subset\, G$ (recall that cocompact (or normal) lattices are those 
for which the quotient is compact). Such a quotient is known to be K\"ahler if and only 
if $G$ is abelian.

All the compact complex parallelizable manifolds admit a holomorphic affine structure 
(connection) given by the trivialization of the holomorphic tangent bundle (by right-invariant 
vector fields). As soon as $G$ is non-abelian the holomorphic affine connection for which 
right-invariant vector fields are parallel have non-vanishing torsion and, consequently, it is 
not flat.

We will prove the following:

\begin{proposition}\label{parallelizable affine} Let $G$ be a complex semi-simple Lie group
and $\Gamma$ a cocompact lattice in $G$. Then the quotient $G/\Gamma$ does not admit any branched
flat affine structure.
\end{proposition}

The following lemma will be needed in the proof of Proposition \ref{parallelizable affine}.

\begin{lemma} \label{regular} Let $G$ be a complex semi-simple Lie group and $\Gamma$ a cocompact
lattice in $G$. Then any branched holomorphic Cartan geometry on
$X\,=\,G / \Gamma$ has an empty branching set.
\end{lemma}

\begin{proof} Assume, by contradiction, that the branching set is not empty. Then, by
Lemma \ref{lem1} the branching set must be a divisor in $X$. On the other hand, it is known,
\cite{HM}, that $G / \Gamma$ contains no divisor, which is a contradiction.
\end{proof}

Now we go back to the proof of Proposition \ref{parallelizable affine}.

\begin{proof}[{Proof of Proposition \ref{parallelizable affine}}]
Assume, by contradiction, that $X\,=\,G / \Gamma$ admits a branched flat affine structure. Using 
Lemma \ref{regular} the branching set must be empty. Consider then the holomorphic affine 
connection $\nabla$ in the holomorphic tangent bundle $TX$ associated the holomorphic flat affine 
structure. If $d$ is the complex dimension of $X$, denote by $(V_1, V_2, \ldots, V_d)$ a family 
of globally defined holomorphic vector fields on $X$ trivializing $TX$ (the $V_i$'s descend from 
right-invariant vector fields on $G$). For any $i,j$, the holomorphic vector field 
$\nabla_{V_i}V_j$ is also globally defined on $X$ and must be a linear combination of $V_i$'s 
with constant coefficients. It now follows that the pull-back of $\nabla_{V_i}V_j$ to $G$ is a 
right-invariant vector field. This implies that the pull-back to $G$ of $\nabla$ is 
right-invariant. But it is known, \cite{D4}, that a semi-simple complex Lie group does not admit 
translation invariant holomorphic flat affine structures, which is a contradiction.
\end{proof} 

The simplest example is that of compact quotients of $\text{SL}(2, \mathbb{C})$ by lattices 
$\Gamma$: they do not admit any branched flat holomorphic affine structure. However, as we will 
see, they admit flat holomorphic projective structures.

Indeed, the Killing quadratic form on the Lie algebra of $\text{SL}(2, \mathbb{C})$ is nondegenerate. It endows the complex manifold $\text{SL}(2, \mathbb{C})$ with a right-invariant {\it holomorphic Riemannian metric} in the sense of the following definition.

\begin{definition}
A holomorphic Riemannian metric on $X$ is a holomorphic section
$$
g\, \in\, H^0(X,\, \text{S}^2((TX)^*))
$$
such that for every point $x\, \in\, X$ the quadratic form $g(x)$ on the fiber $T_xX$
is nondegenerate. 
\end{definition}

A holomorphic Riemannian metric on a complex manifold of dimension $d$ is a holomorphic Cartan 
geometry of the type $G/H$, where $H$ is the complex orthogonal group $\text{O}(d, \mathbb{C})$ 
and $G$ is the semi-direct product $\mathbb{C}^d \rtimes \text{O}(d, \mathbb{C})$ for the 
standard action of $\text{O}(d, \mathbb{C})$ on $\mathbb{C}^d$ \cite[Ch.~6]{Sh}.

As in the Riemannian and pseudo-Riemannian setting, one associates to a holomorphic 
Riemannian metric $g$ a unique holomorphic affine connection $\nabla$. This 
connection $\nabla$, called the Levi--Civita connection for $g$, is uniquely 
determined by the following two properties:
\begin{itemize}
\item $\nabla$ is torsionfree, and

\item the holomorphic tensor $g$ is parallel with respect to $\nabla$.
\end{itemize}
The curvature of this Levi--Civita connection $\nabla$ vanishes identically if and only if $g$ is 
locally isomorphic to the standard flat model $dz_1^2 + \ldots + dz_n^2$, seen as a homogeneous 
space for the group $G\,=\, \mathbb{C}^d \rtimes \text{O}(d, \mathbb{C})$.

The holomorphic Riemannian metric on $\text{SL}(2, {\mathbb C})$ coming from the Killing 
quadratic form is bi-invariant (since the Killing quadratic form is invariant under the adjoint 
action of $\text{SL}(2, {\mathbb C})$). It has nonzero constant sectional curvature \cite{Gh}. 
Since the Levi--Civita connection of a metric of constant sectional curvature is known to be 
projectively flat, this endows $\text{SL}(2, {\mathbb C})$ with a bi-invariant flat holomorphic 
projective structure. For more details about the geometry of holomorphic Riemannian metrics one 
can see \cite{Gh,D,DZ}.

Interesting exotic deformations of parallelizable manifolds $\text{SL}(2, 
\mathbb{C})/ \Gamma$ were constructed by Ghys in~\cite{Gh}.

The above mentioned deformations in \cite{Gh} are constructed by choosing a group homomorphism 
$$u\,:\, \Gamma \,\longrightarrow\, \text{SL}(2, \mathbb{C})$$ and considering the embedding $ 
\gamma \,\longmapsto\, (u(\gamma),\, \gamma)$ of $\Gamma$ into $\text{SL}(2, \mathbb{C}) \times 
\text{SL}(2, \mathbb{C})$ (acting on $\text{SL}(2, \mathbb{C}$) by left and right translations). 
Algebraically, the action is given by:
$$
(\gamma,\,x) \,\in\, \Gamma \times \text{SL}(2,\mathbb{C})\,\longrightarrow\,
u(\gamma^{-1}) x \gamma\,\in\, \text{SL}(2,\mathbb{C})\, .
$$
 
It is proved in \cite{Gh} that, for $u$ close enough to the trivial morphism, 
$\Gamma$ acts properly and freely on $\text{SL}(2, \mathbb{C})$ and the 
quotient $M(u, \Gamma)$ is a compact complex manifold (covered by $\text{SL}(2, 
\mathbb{C)}$). In general, these examples do not admit parallelizable manifolds as 
finite covers. Moreover, for generic $u$, the space of all holomorphic global vector 
fields on them is trivial. All manifolds $M(u, \Gamma)$ inherit a flat holomorphic 
projective structure (coming from the bi-invariant projective structure constructed 
above). Moreover, any small deformation of the manifold $\text{SL}(2, \mathbb{C})$ 
is isomorphic to $M(u, \Gamma)$ for some $u$ \cite{Gh}.

Therefore, we get the following:

\begin{theorem}[Ghys]
Complex compact parallelizable manifolds ${\rm SL}(2, {\mathbb C}) / \Gamma$ and their small
deformations admit a flat holomorphic projective structure.
\end{theorem}

It is not not known whether for generic homomorphisms $u$, complex manifolds $M(u, \Gamma)$ admit 
any other flat holomorphic projective structure apart from the standard one (that descends from the 
bi-invariant flat holomorphic projective structure on $\text{SL}(2,\mathbb{C})$ constructed 
above).

For some non-generic homomorphisms $u$, complex manifolds $M(u, \Gamma)$ also admit holomorphic 
Riemannian metrics with nonconstant sectional curvature \cite{Gh}. The associated holomorphic 
projective structures on those manifolds are not flat.

Recall here the main result in \cite{DZ}: 

\begin{theorem}[{\cite{DZ}}]
Let $M$ be a compact complex threefold endowed with a holomorphic Riemannian metric. Then $M$
admits a finite unramified covering bearing a holomorphic Riemannian metric of constant sectional
curvature (and hence has the associated flat holomorphic projective structure). 
\end{theorem}

Now we will describe the global geometry of holomorphic projective structures on complex parallelizable manifolds. Let us first prove the following.

\begin{lemma} \label{lemma can} Consider a holomorphic projective connection on a compact complex manifold $X$ with trivial canonical bundle. Then $X$ admits a holomorphic affine connection $\nabla$ which is 
projectively isomorphic to the given holomorphic projective connection.
\end{lemma}

\begin{proof}
Let $X\,= \,\bigcup U_i$ be an open covering of $X$ such that on each $U_i$ there exists a
holomorphic affine connection $\nabla_i$ projectively equivalent to the given projective 
connection. Let $\omega$ be a global nontrivial holomorphic section of the canonical bundle
(it is trivial by assumption). On 
each $U_i$, there exists a unique holomorphic affine connection $\widetilde{\nabla}_i$ 
projectively equivalent to $\nabla_i$ satisfying the condition that $\omega$ is parallel with
respect to $\widetilde{\nabla}_i$ \cite[Appendix A.3]{OT}. By uniqueness, these $\widetilde{\nabla}_i$'s 
agree on the overlaps of the $U_i$'s and define a global holomorphic affine connection on $X$ 
projectively equivalent to the original holomorphic projective connection (for a different proof 
one can also combine two results in \cite[p. 96]{Gu} and \cite[p. 78--79]{KO}).
\end{proof} 

The following proposition is proved using Lemma \ref{lemma can}.

\begin{proposition}\label{proj paral} 
Let $G$ be a complex Lie group of dimension $d$ and $\Gamma$ a lattice in $G$. Then $X\,=\, G/ 
\Gamma$ admits a flat holomorphic projective structure if and only if there exists a Lie group 
homomorphism $i\,:\, \widetilde G \,\longrightarrow\, {\rm PGL}(d+1,{\mathbb C})$, where 
$\widetilde G$ is the universal cover of $G$, such that $i(\widetilde{G})$ acts with an open 
orbit on the standard model ${\mathbb C}P^d$.
\end{proposition}

Note that the condition in the statement of Proposition \ref{proj paral} is equivalent to the 
existence of a Lie algebra homomorphism $\widetilde i$ from the Lie algebra of $G$ into the Lie 
algebra of $\text{PGL}(d+1,{\mathbb C})$, such that the image of $\widetilde i$ intersects 
trivially the Lie subalgebra of the stabilizer $Q$ of a point in ${\mathbb C}P^d$. A 
classification of those complex Lie algebras admitting such a homomorphism is done in \cite{K} 
(see also \cite{A} for the real case).

\begin{proof}[{Proof of Proposition \ref{proj paral}}]
First assume that there exists a group homomorphism $i\,:\, \widetilde{G}
\,\longrightarrow\,\text{PGL}(d+1,{\mathbb C})$ such that $i(\widetilde G)$ acts on
${\mathbb C}P^d$ with an open orbit $O \,\subset\, X$. Fix a point $o \,\in\, O$, and consider the map
$$ \pi\,:\, \widetilde G \,\longrightarrow\, O$$ defined by $\pi(g) \,=\, i(g) \cdot o$ for
all $g \,\in\, \widetilde G$. This map $\pi$ is a covering and the pull-back of the flat
holomorphic projective structure on $O$ through $\pi$ is a right-invariant flat holomorphic
projective structure on $\widetilde G$. This flat holomorphic projective structure on
$\widetilde G$ descends to
the quotient $X\,=\,\widetilde G / \widetilde \Gamma$, where $\widetilde \Gamma$ is the
inverse image of $\Gamma$ in the universal covering $\widetilde G$ of $G$.

To prove the converse, assume that $G / \Gamma$ is equipped with
a flat holomorphic projective structure. By Lemma \ref{lemma can}, there exists a holomorphic affine connection $\nabla$ on $G / \Gamma$ which is projectively equivalent 
to the given flat holomorphic projective structure. The proof of Proposition \ref{parallelizable affine} shows that the pull-back of $\nabla$ to $G$ is a right-invariant holomorphic affine connection. In particular,
the pull-back of the initial flat holomorphic projective structure to $G$ is right-invariant. It follows that the Lie algebra of $G$ acts locally projectively on the standard projective model ${\mathbb C}P^d$. Since the model is simply connected, this local action extends to a projective locally free global action of $\widetilde G$ on $\text{PGL}(d+1,{\mathbb C})$. This gives the required Lie group homomorphism $i$.
\end{proof}

It may be remarked that the Lie group homomorphism $i$ in the statement of Proposition \ref{proj 
paral} extends the monodromy homomorphism $\rho \,:\, \widetilde {\Gamma}\,\longrightarrow\, 
\text{PGL}(d+1,{\mathbb C})$ to a Lie group homomorphism $i$. The projective structures with
this property are called {\it homogeneous}.

In order to see that $\text{SL}(2, {\mathbb C})$ admits actions as in the statement of 
Proposition \ref{proj paral}, consider the irreducible linear action of $\text{SL}(2, {\mathbb C})$
on the vector space of homogeneous polynomials of degree $3$ in two variables (by linearly changing 
the variables). The projectivization of this linear action gives a projective action of 
$\text{SL}(2, {\mathbb C})$ on ${\mathbb C}P^3$ with an open orbit, namely the $\text{SL}(2, 
{\mathbb C})$-orbit of those polynomials that are product of three distinct linear forms 
(recall that the projective action of $\text{SL}(2, {\mathbb C})$ on the projective line 
${\mathbb C}P^1$ is transitive on the set of triples of distinct points).

\section{Calabi--Yau manifolds and branched Cartan geometries} \label{Calabi--Yau}

In this section we are interested in understanding branched holomorphic Cartan geometries on 
Calabi--Yau manifolds.

Recall that K\"ahler Calabi--Yau manifolds are compact complex K\"ahler manifolds 
with the property that the first Chern class (with real coefficients) of the 
holomorphic tangent bundle vanishes. By Yau's theorem proving Calabi's conjecture, 
those manifolds admit K\"ahler metrics with vanishing Ricci curvature \cite{Ya}. 
Compact K\"ahler manifolds admitting a holomorphic affine connection have vanishing 
real Chern classes \cite{At}; it was proved in \cite{IKO} using Yau's result that 
they must admit finite unramified coverings which are complex tori.

It was proved in \cite{BM} (see also \cite{D1,D3}) that Calabi--Yau manifolds bearing a holomorphic 
Cartan geometry admit finite unramified covers by complex tori. We extend here this result to 
branched holomorphic Cartan geometries.

\begin{theorem}
A compact (K\"ahler) Calabi--Yau manifold $X$ bearing a branched holomorphic
affine structure admits a finite unramified covering by a complex torus.
\end{theorem}

\begin{proof}
Since $c_1(TX)\,=\, 0$, part (ii) in Corollary \ref{corollaire deg} implies that the branched holomorphic 
affine structure on $X$ is actually a holomorphic affine structure (connection). Hence $X$ admits a finite unramified covering by a
complex torus \cite{IKO}.
\end{proof}

\begin{theorem}\label{thm2}
Let $X$ be a compact simply connected K\"ahler manifold such that $c_1(TX)\,=\, 0$.
Let $E$ be a holomorphic vector bundle on $X$ equipped with a holomorphic connection.
Then $E$ is a trivial holomorphic vector bundle and $D$ is the trivial connection on it.
\end{theorem}

\begin{proof} 
The theorem of Yau says that $X$ admits a Ricci--flat K\"ahler metric \cite{Ya}. Fix a Ricci--flat K\"ahler
form $\omega$ on $X$. The degree of a torsionfree coherent analytic sheaf on $X$ will be defined
using $\omega$ (as in \eqref{deg}). From the given condition that $\omega$ is Ricci--flat
it follows that the tangent bundle $TX$
is polystable. Since $TX$ is polystable with $c_1(TX)\,=\, 0$, and $E$ admits a holomorphic
connection, it follows that $E$ is semistable \cite[p.~2830]{Bi}.

Note that $c_i(E)\,=\, 0$, $i\, \geq\, 1$, because $E$ admits a holomorphic connection
\cite[p.~192--193, Theorem~4]{At}. In particular, we have $\text{degree}(E)\,=\, 0$.

We will now recall a theorem of Simpson in \cite{Si}. Let $(Y,\,\omega_Y)$ be a compact
K\"ahler manifold of dimension $m$. The works of Corlette and Simpson,
\cite{Co}, \cite{Si0}, give a natural bijective correspondence between the complex vector
bundles on $Y$ with irreducible flat connection
and stable Higgs bundles $(V,\, \varphi)$ on $Y$ with
$\text{degree}(V)\,=\, 0\,=\, ch_2(V)\wedge \omega^{m-2}_Y$ (see \cite[p.~20, 
Corollary~1.3]{Si}). It
should be mentioned that if $(U,\, \mathcal{D})$ is a complex vector bundle on $Y$ with an irreducible flat
connection, and $(V, \, \varphi)$ is the polystable Higgs bundles on $Y$ corresponding to it, then
the holomorphic vector bundles on $Y$ underlying $U$ and $V$ need not coincide in general. They
do coincide when $\varphi\,=\, 0$. It should be mentioned that $\varphi\,=\, 0$ if and only if
the corresponding flat
connection $\mathcal{D}$ is unitary. In \cite{Si}, Simpson extended this correspondence to connections
which are not necessarily irreducible and Higgs bundles not necessarily polystable. He proved
an equivalence of categories between the following two:
\begin{enumerate}
\item The category of complex vector bundles $U$ on $Y$ with a flat connection $\mathcal{D}$.

\item The category of semistable Higgs bundles $(V, \,\varphi)$ on $Y$
with $\text{degree}(V)\,=\, 0\,=\, ch_2(V)\wedge \omega^{m-2}_Y$ and satisfying the
condition that $V$ admits a filtration of holomorphic subbundles such that each 
subbundle in the filtration is preserved by $\varphi$, and each successive quotient
for this filtration equipped with the Higgs field induced by $\varphi$ is polystable with
degree zero.
\end{enumerate}
(See \cite[p.~36, Lemma~3.5]{Si}.) When $Y$ is a complex projective manifold,
and the cohomology class of $\omega_Y$ is rational, Simpson improved
the above equivalence. For a complex projective polarized manifold $Y$ there is an equivalence of
categories between the following two:
\begin{enumerate}
\item The category of complex vector bundles $U$ on $Y$ with a flat connection $\mathcal{D}$.

\item The category of semistable Higgs bundles $(V, \,\varphi)$ on $Y$ with 
$\text{degree}(V)\,=\, 0\,=\, ch_2(V)\wedge \omega^{m-2}_Y$.
\end{enumerate}
(See \cite[p.~40, Corollary~3.10]{Si}.) In both these equivalences of
categories the holomorphic vector bundles underlying $U$ and $V$ do not coincide
in general. But they indeed coincide if $\varphi\,=\, 0$.

Therefore, setting $\varphi\, =\, 0$ in the first equivalence of categories we get
the following:

A holomorphic vector bundle $U$ on a compact K\"ahler manifold $(Y,\,\omega_Y)$
admits a flat holomorphic
connection if $U$ admits a filtration of holomorphic subbundles
such that each successive quotient $Q$ for the filtration
is polystable with $\text{degree}(Q)\,=\, 0\,=\, ch_2(Q)\wedge \omega^{m-2}_Y$.

Similarly, setting $\varphi\, =\, 0$ in the second equivalence of categories we get
the following:

A holomorphic vector bundle $U$ on a complex polarized projective manifold admits a flat 
holomorphic connection if $U$ is semistable with $\text{degree}(U)\,=\, 0\,=\, 
ch_2(U)\wedge \omega^{m-2}_Y$.

Consequently,
if the Calabi--Yau manifolds $X$ in the theorem is
{\it projective}, and the cohomology class of $\omega$ is rational,
then that $E$ admits a flat holomorphic connection,
because $E$ is semistable with vanishing Chern classes. Since $X$ is simply connected
all flat bundles on $X$ are trivial. Therefore, $E$ is the trivial vector bundle.
Since $H^0(X, \, \Omega^1_X)\,=\, 0$, the trivial holomorphic vector bundle has exactly
one holomorphic connection, namely the trivial connection. Hence
$D$ is the trivial connection on $E$.

We will now address the general K\"ahler case.

Let $V\, \subset\, E$ be a polystable subsheaf such that
\begin{itemize}
\item $\text{degree}(V)\,=\, 0$, and

\item the quotient $E/V$ is torsionfree.
\end{itemize}
The second condition implies that $V$ is reflexive. Since $E$ is semistable, and $V$ is polystable
with $\text{degree}(V)\,=\, 0\,=\, \text{degree}(E)$, it follows that $E/V$ is semistable
with $\text{degree}(E/V)\,=\, 0$.

Let $d$ be the complex dimension of $X$. Let the ranks of $V$ and $E/V$ be $r$ and $s$ respectively. 
Since $V$ and $E/V$ are semistable, we have the Bogomolov inequality
\begin{equation}\label{g1}
((2r\cdot c_2(V)-(r-1)c_1(V)^2)\cup \omega^{d-2})\cap [X]\, \geq\, 0\, ,
\end{equation}
\begin{equation}\label{g2}
((2s\cdot c_2(E/V)-(s-1)c_1(E/V)^2)\cup \omega^{d-2})\cap [X]\, \geq\, 0
\end{equation}
\cite[Lemma 2.1]{BM}.

We will show that the inequalities in \eqref{g1} and \eqref{g2} are equalities.
Denote the sheaf $E/V$ by $W$. We have
$$
2(r+s)c_2(V\oplus W)- (r+s-1)c_1(V\oplus W)^2
$$
$$
=\,
2(r+s)(c_2(V)+c_2(W)+c_1(V)c_1(W))- (r+s-1)(c_1(V)^2+c_1(W)^2+2c_1(V)c_1(W))
$$
$$
=\, \frac{r+s}{r}(2rc_2(V)-(r-1)c_1(V)^2)+\frac{r+s}{s}(2rc_2(W)-(s-1)c_1(W)^2)
$$
$$ 
-\frac{s}{r}c_1(V)^2-\frac{r}{s}c_1(W)^2 + 2c_1(V)c_1(W)
$$
$$
\,=\,\frac{r+s}{r}(2rc_2(V)-(r-1)c_1(V)^2)+\frac{r+s}{s}(2rc_2(W)-(s-1)c_1(W)^2) 
-\frac{1}{sr}(s\cdot c_1(V)-r\cdot c_1(W))^2\, .
$$
On the other hand, $c_i(V\oplus W)\,=\, c_i(E)\,=\, 0$, so
$$
\frac{r+s}{r}((2r\cdot c_2(V)-(r-1)c_1(V)^2)\cup \omega^{d-2})\cap [X]
+\frac{r+s}{s}((2rc_2(W)-(s-1)c_1(W)^2)\cup \omega^{d-2})\cap [X]
$$
\begin{equation}\label{g3}
-
\frac{1}{sr}((s\cdot c_1(V)-r\cdot c_1(W))^2 \cup \omega^{d-2})\cap [X]
\,=\,0\, .
\end{equation}
{}From Hodge index theorem (see \cite[Section 6.3]{Vo}) it follows that
$$
-\frac{1}{sr}((s\cdot c_1(V)-r\cdot c_1(W))^2 \cup \omega^{d-2})\cap [X]
\, \geq\, 0\, .
$$
Therefore, from \eqref{g3} we conclude that the inequalities in \eqref{g1} and \eqref{g2}
are equalities.

Since $((2r\cdot c_2(V)-(r-1)c_1(V)^2)\cup \omega^{d-2})\cap [X]\, =\, 0$, from
\cite[p.~40, Corollary~3]{BS} we conclude that $V$ is a polystable {\it vector bundle}
admitting a projectively flat unitary connection. Therefore, projective bundle $P(V)$
for $V$ is given by a representation of $\pi_1(X)$ in $\text{PU}(r)$. As $X$ is simply
connected, we conclude that the projective bundle $P(V)$ is trivial. Hence
\begin{equation}\label{V}
V\,=\, L^{\oplus r}\, ,
\end{equation}
where $L$ is a holomorphic line bundle on $X$. We have
$$
\text{degree}(L)\,=\, 0\, ,
$$
because $\text{degree}(V)\,=\, 0$.

Now assume that $V$ is preserved by the connection $D$ on $E$. Then $V$ is a 
subbundle of $E$, and the quotient $E/V$ has a holomorphic connection $D_1$ induced 
by $D$. Consequently, we may repeat the above arguments for $(E/V,\, D_1)$ and get a 
subsheaf $V_1\, \subset\, E/V$ which is a direct sum of line bundles of degree zero 
(as in \eqref{V}). Again assume that $V_1$ is preserved by $D_1$ and repeat the 
argument. In this way we get a filtration of $E$ by subbundles such that each 
successive quotient is a polystable vector bundle of degree zero. As explained 
before, a theorem of Simpson implies that $E$ admits a flat holomorphic connection. 
Since $X$ is simply connected, this implies that $E$ is holomorphically trivial. A 
trivial holomorphic vector bundle on $X$ has exactly one holomorphic connection 
because $H^0(X, \, \Omega^1_X)\,=\, 0$ (recall that $X$ is simply connected). 
Therefore, a trivial holomorphic vector bundle on $X$ has only the trivial 
connection.

Now assume the opposite, namely that $V$ is not preserved by the connection $D$ on 
$E$. Consider the holomorphic section of $\text{Hom}(V,\, E/V)\otimes \Omega^1_X$ 
given by $D$; it is nonzero because $V$ is not preserved by $D$. Let
\begin{equation}\label{de1}
\delta\, :\, TX\, \longrightarrow\, \text{Hom}(V,\, E/V)
\end{equation}
be the homomorphism given by this section.

The rank of $\text{Hom}(V,\, E/V)$ is $rs$. We have
\begin{equation}\label{dz}
\text{degree}(\text{Hom}(V,\, E/V))\,=\, r\cdot
\text{degree}(E/V)-s\cdot \text{degree}(V)\,=\, 0\, .
\end{equation}
We have $V^*$ to be semistable because $V$ is semistable. Now, since both $V^*$ and $E/V$ 
are semistable, it follows that $\text{Hom}(V,\, E/V)\,=\, (E/V)\otimes V^*$ is 
semistable \cite[Lemma 2.7]{AB}; recall that $V$ is locally free, so $(E/V)\otimes V^*$ is
torsionfree. Since $\text{Hom}(V,\, E/V)$ is semistable of
degree zero (shown in \eqref{dz}), and $TX$ is a polystable vector bundle of degree zero, we conclude that
the image $\delta(TX)$ in \eqref{de1} is also a polystable vector bundle of degree zero;
here we are using the fact that the image of a polystable sheaf in a semistable sheaf of
same slope (=\,degree/rank) is also polystable of the common slope.

Let $t$ be the rank of $U\,:=\, \delta(TX)$. We have
$$
(2rs\cdot c_2(\text{Hom}(V,\, E/V))-(rs-1)c_1(\text{Hom}(V,\, E/V))^2)
\cup \omega^{d-2})\cap [X]
$$
$$
=\, ((2r\cdot c_2(V)-(r-1)c_1(V)^2)\cup \omega^{d-2})\cap [X]
$$
$$
+
((2s\cdot c_2(E/V)-(s-1)c_1(E/V)^2)\cup \omega^{d-2})\cap [X]\,=\, 0\, .
$$
This implies that
$$
((2t\cdot c_2(U)-(t-1)c_1(U)^2)\cup \omega^{d-2})\cap [X]\,=\, 0\, ,
$$
because the Bogomolov inequality holds for both $U$ and $\text{Hom}(V,\, E/V)/U$. Indeed,
the Bogomolov inequality holds for all three terms in the short exact sequence
$$
0\, \longrightarrow\, U\, \longrightarrow\,\text{Hom}(V,\, E/V) \, \longrightarrow\,
\text{Hom}(V,\, E/V)/U \, \longrightarrow\, 0
$$
and furthermore it is an equality for $\text{Hom}(V,\, E/V)$; hence the
Bogomolov inequality is an equality for both $U$ and $\text{Hom}(V,\, E/V)/U$.

Again from \cite[p.~40, Corollary~3]{BS} we conclude that $P(U)$ admits a flat connection.
Hence $U$ is of the form
$$
U\,=\, N^{\oplus t}\, ,
$$
where $N$ is a holomorphic line bundle on $X$ of degree zero.

Since $TX$ is polystable, the quotient bundle $U$ is a direct summand of $TX$.
This implies that $U$ is a subbundle of $TX$. Hence we have a holomorphic
decomposition
\begin{equation}\label{dec}
TX\, =\, N\oplus N'\, ,
\end{equation}
where $N$ is a holomorphic line bundle on $X$ of degree zero, and the rank of $N'$
is $d-1$.

A result of Beauville \cite[Theorem A]{Be} associates to any holomorphic splitting 
$$TX \,=\,U_1 \oplus U_2 \oplus \ldots \oplus U_j$$ a corresponding decomposition $X 
\,=\,X_1 \times X_2 \times \ldots \times X_j$, with $X_i$ simply connected Calabi--Yau 
manifolds, such that $U_i \,=\, \pi_{i}^{*}(TX_i)$, where $p_i \,:\, X\,
\longrightarrow\, X_i$, $1\,\leq\, i\, \leq\, j$, are the canonical
projections. Now from \eqref{dec} we conclude that $X$
is a product of Calabi--Yau manifolds with one factor of dimension one. But there is
no simply connected compact Calabi--Yau manifold of complex dimension one. Therefore, we get a
contradiction. This completes the proof.
\end{proof}

\begin{corollary} \label{final}\mbox{}
\begin{enumerate}
\item[(i)] Any branched holomorphic Cartan geometry of type $G/H$, with $G$ complex affine Lie group, on a compact simply connected (K\"ahler) 
Calabi--Yau manifold is flat. Consequently, the model $G/H$ of the Cartan geometry must be 
compact.

\item[(ii)] Non-projective compact simply connected (K\"ahler) Calabi--Yau manifolds do not admit branched holomorphic
projective structures.
\end{enumerate}
\end{corollary}

Recall that $G$ is {\it a complex affine} Lie group if it admits 
a linear representation $\rho \,:\, G \,\longrightarrow\, {\rm GL}(N, {\mathbb C})$, for
some $N$, with discrete kernel. Complex semi-simple
Lie groups are complex affine (see Theorem 3.2, chapter XVII, in \cite{Ho}). 

\begin{proof}
Let $X$ be a simply connected Calabi--Yau manifold endowed with a branched 
holomorphic Cartan geometry of type $G/H$, with $G$ complex affine Lie group.

Let $\rho \,:\, G\,\longrightarrow\, {\rm GL}(N, {\mathbb C})$ be a linear representation
of $G$ with discrete kernel. The corresponding Lie algebra representation $\rho'\,:\,
\text{Lie}(G) \,\longrightarrow\,{\rm M}(N, {\mathbb C})$ is an injective.

Consider the principal bundle $E_G$ and the associated holomorphic vector bundle 
$E_{\rho}$ of rank $N$ on $X$ for the above homomorphism
$\rho$. Then $E_{\rho}$ inherits a holomorphic 
connection $\theta'_{\rho}$ and, by Theorem \ref{thm2}, this connection 
$\theta'_{\rho}$ must be flat. Since the curvature of $\theta'_{\rho}$ is the image 
of the curvature of the connection $\theta'$ of $E_G$ through $\rho'$, and $\rho'$ is 
injective, it follows that $\theta'$ is also flat.

Proof of (i):\ As shown above, Theorem \ref{thm2} implies that the associated holomorphic connection $\theta'$ of $E_G$ 
must be flat. Consequently, the Cartan geometry of type $G/H$ is flat. The developing map $\delta \,:\, X 
\,\longrightarrow\, G/H$ is a branched holomorphic map. This implies that $\delta(X)=G/H$ is 
compact.

Proof of (ii):\ This follows from part (i) and Corollary \ref{corollary simply connected}.
\end{proof}

\section*{Acknowledgements}

We thank the referee for pointing out \cite{Br}. This work was carried out while the first 
author was visiting the Universit\'e C\^ote d'Azur. He thanks the Universit\'e 
C\^ote d'Azur and Dieudonn\'e Department of Mathematics for their hospitality.


\end{document}